\documentclass[a4paper,11pt,reqno]{article}
\usepackage{times,amssymb,amsmath,amsthm,paralist,url,anysize,nicefrac,amsfonts} 
\usepackage[english]{babel}
\usepackage[utf8]{inputenc}
\usepackage[T1]{fontenc}
\usepackage{graphicx} 
\usepackage{mathtools} 
\usepackage[font=small]{caption}

\newcommand\Z{{\mathbb Z}} 
\newcommand\R{{\mathbb R}} 
\newcommand\eps{\varepsilon} 
\newcommand{\ceil}[1]{\lceil #1\rceil}
\newcommand{\floor}[1]{\lfloor #1\rfloor} 
\DeclareMathOperator{\conv}{conv} 
\DeclareMathOperator{\excess}{excess} 
 
\newtheorem{theorem}{Theorem} 
\newtheorem{lemma}[theorem]{Lemma}

\newtheorem{proposition}[theorem]{Proposition}  
\newtheorem{conjecture}[theorem]{Conjecture}

\usepackage{tikz}
	\usetikzlibrary{shapes,shadows,arrows,patterns,trees,automata,positioning,calc,backgrounds,fadings}
\usepackage[colorlinks=false,urlcolor=blue]{hyperref}
\usepackage[margin=3cm, includefoot, includehead]{geometry}
  
\title{Convex Equipartitions of Colored Point Sets} 

\author{Pavle V. M. Blagojevi\'c%
\footnote{The authors are supported by DFG via
the Collaborative Research Center TRR~109 ``Discretization in Geometry and Dynamics.''
PVMB is also supported by the grant ON 174008 of the Serbian Ministry of Education and Science.
GMZ is also supported by DFG via the Berlin Mathematical School BMS.}\\
Institut f\"ur Mathematik, FU Berlin\\Arnimallee 2, 14195 Berlin, Germany\\
\url{blagojevic@math.fu-berlin.de} 
\and
G\"unter Rote\\
Institut f\"ur Informatik, FU Berlin\\Takustraße 9, 14195 Berlin, Germany\\
\url{rote@inf.fu-berlin.de}
\and
Johanna K. Steinmeyer\\ 
Institut f\"ur Mathematik, FU Berlin\\Arnimallee 2, 14195 Berlin, Germany\\
\url{steinmeyer@math.fu-berlin.de}
\and
G\"unter M. Ziegler\\ 
Institut f\"ur Mathematik, FU Berlin\\Arnimallee 2, 14195 Berlin, Germany\\
\url{ziegler@math.fu-berlin.de}}
\date{May 23, 2017; revised December 11, 2017}

\begin{document} 

\maketitle

\begin{abstract}
We show that any $d$-colored set of points in general position in~$\R^d$ can be partitioned into $n$ subsets with disjoint convex hulls such that the set of points and all color classes are partitioned as evenly as possible.	
This extends results by Holmsen, Kyn\v{c}l \& Valculescu (2017) and establishes a special case of their general conjecture.
Our proof utilizes a result obtained independently by Soberón and by Karasev in 2010, on simultaneous
equipartitions of $d$ continuous measures in~$\R^d$ by $n$ convex
regions. This gives a convex partition of $\R^d$ with the desired properties, except
that points may lie on the boundaries of the regions.
In order to resolve the ambiguous assignment of these points, we set up a network flow problem.
The equipartition of the continuous measures gives a fractional flow.
The existence of an integer flow then yields the desired partition of the point set.
\end{abstract}

\section{Introduction}
A (finite) set $X$ of points in $\R^d$ is in \emph{general position} if every subset of size at most $d+1$ is affinely independent.
A partition $X=X_1\sqcup\dots\sqcup X_m$ of $X$ into $m$ disjoint subsets is an \emph{$m$-coloring} of $X$.
The sets $X_1,\ldots,X_m$ are called \emph{color classes} and we say that the set $X$ is 
{\emph{$m$-colored}}. 
A subset $Y\subseteq X$ containing points from at least $j$ distinct color classes 
is said to be \emph{$j$-colorful}.  

In this language, the classical partition result of Akiyama and Alon reads as follows.

\begin{theorem}[Akiyama--Alon \cite{AA89}]\label{AA}
	Let $n,d$ be positive integers, and let $X$ be a $d$-colored set of points in general position in $\R^d$, 
	with each color class containing $n$ points. 
	Then there is a partition of $X$ into $n$ $d$-colorful sets of size $d$ whose convex hulls are pairwise disjoint. 
\end{theorem}

Akiyama and Alon gave a beautifully simple proof using a discrete
version of the ham-sandwich theorem, which is a well known consequence of the Borsuk--Ulam theorem.
The use of such
topological methods created a lot of progress in solving discrete partitioning problems.
In fact, many related  
partition results have both a continuous mass partition as well as a discrete colored version\textemdash often equivalent.

In this paper, we consider the following conjecture of  Holmsen, Kyn\v{c}l and Valculescu {\cite[Conjecture~3]{HKV16}}.

\begin{conjecture}[Holmsen--Kyn\v{c}l--Valculescu, 2016]\label{conjecture} 
	Let $m,k,n$, and $d$ be positive integers, and let $X$ be an $m$-colored set of $kn$ points in general position in
	$\R^d$. 
	Suppose there is a partition of~$X$ into $n$ $d$-colorful sets of size $k$.
	Then there is also such a partition with the additional geometric property that
	the convex hulls of the $n$ sets are pairwise disjoint.
\end{conjecture}
Here, the assumption of the existence of a partition
depends only on the number of color classes and their sizes,
and it involves no geometry.
It is obviously a necessary condition. In particular, it implies that $m\ge d$ and $k\ge d$. 

\autoref{AA} answers the case when $k=m=d$.
The case $m\geq k=d=2$ was settled by Aichholzer et al.~\cite{AetAl10} and by Kano, Suzuki and Uno \cite{KSU14}.
Further developments on the planar case were made independently by
Bespamyatnikh, Kirkpatrick and Snoeyink \cite{BKS00}, Ito, Uehara and
Yokoyama \cite{IUY00} as well as Sakai \cite{Sakai02}, who confirmed the conjecture for two colors ($m=d=2$) when the sizes of the color classes are divisible by $n$.
Holmsen, Kyn\v{c}l and Valculescu resolved the conjecture for the remaining cases in the plane, as well as for the case when $k-1=m=d\geq 2$, the latter by giving a particular discretization of the ham-sandwich theorem~\cite{HKV16}. 
Their method is similar to the one used previously by Kano and Kyn\v{c}l~\cite{KK16} to establish the case $m-1=d=k$, who for the proof developed a generalization of the ham-sandwich theorem for $d+1$ measures in $\R^d$, which they called the hamburger theorem.

Holmsen et al.\ emphasized the connection of the conjecture with a continuous analogue for the case $m=d$, proved in the plane by Sakai \cite{Sakai02} and extended to arbitrary dimension by Sober\'on \cite{S12} and independently by Karasev \cite{K11}.
(A~more general version, for functions that are not necessarily measures, was obtained soon after by Karasev, Hubard and Aronov \cite{KHA13} and by Blagojevi\'c and Ziegler \cite{BZ14}.) 

\begin{theorem}[Sober\'on--Karasev, 2010]\label{Spatzenbraten}
	Let $n,d$ be positive integers, and let $\mu_1,\dots,\mu_d$ be absolutely continuous finite measures on $\R^d$ with respect to the Lebesgue measure. 
	Then there exists a partition of\/~$\R^d$ into $n$ convex regions $C_1,\dots,C_n$ that simultaneously equipartitions all $d$ measures, that is,
	\[
	\mu_i(C_j)=\frac{1}{n}~\mu_i(\R^d)
	\]
	for all $i\in\{1,\dots,d\}$ and all $j\in\{1,\dots,n\}$.
\end{theorem}

Holmsen, Kyn\v{c}l and Valculescu state: 
\begin{quote}
{``However, going from the continuous version to the discrete version seems to require, in many cases, a non-trivial approximation argument, and we do not see how the continuous results~[\ldots] could be used to settle our Conjecture~3 for the case $m=d$.''}
%%HKV say: However, going from the continuous version to the discrete version seems to require, in many cases, a non-trivial approximation argument, and we do not see how the continuous results~\cite{blagojevic,spicy,soberon} could be used to settle our Conjecture~\ref{conj_main} for the case $m=d$.
\end{quote}
Indeed, this is not straightforward. 
However, in this paper we show how it can be done: We confirm
Conjecture~\ref{conjecture} when $m=d$, as a direct corollary of the following main result. 
For this we say that a partition of a finite set $A$ into $n$ parts is an \emph{equipartition} if  each of the parts contains $\lceil\frac{|A|}{n}\rceil$ or $\lfloor\frac{|A|}{n}\rfloor$ elements of $A$.

\begin{theorem}\label{equipart}
	Let $n,d$ be positive integers, and let $X$ be a $d$-colored set of points in general position in~$\R^d$. Then there exists an equipartition of $X$ into $n$ subsets which simultaneously equipartition each of the color classes and whose convex hulls are pairwise disjoint.
\end{theorem}

To see that Theorem~\ref{equipart} implies Conjecture~\ref{conjecture} for the case $m=d$, observe that in this case the condition on $X$ of admitting a partition into $n$ pairwise disjoint $d$-colorful sets of size $k$ implies that each color class has at least $n$ elements. 
In an equipartition of a color class $X_i$, each part contains at
least $\floor{\frac{|X_i|}{n}}\geq 1$ points.
Thus, each part of $X$ contains all $d$ colors.
With $|X|=kn$ and an equipartition of $X$, we get $n$ sets of size $k$ that each contain at least one point of each of the $d$~colors.

\section{Preliminaries} 
\label{toolsection}

In order to discretize Theorem~\ref{Spatzenbraten}, we start by employing a classical idea 
(see Alon \& Akiyama \cite{AA89}):
We replace the points in $X$ with small enough closed balls and then define measures on these. 
The problem with applying the continuous result
is that the boundaries of the regions may cut through some balls,
see Figure~\ref{fig:small_example} (left). 
We will assign every such ``ambiguous'' point to one of the regions intersected by the ball centered at the point.

The following lemma shows that, if the radius $\varepsilon$ of the balls is
small enough, we will always get a partition of $X$ with disjoint
convex hulls, no matter how we resolve the ambiguities. 
In Section~\ref{main-proof} we will prove that we can resolve these ambiguities
in such a way that we get an equipartion of the full point set~$X$ as well as of each of the color classes $X_i$.
Note that this does {not} guarantee the existence of a \emph{partition of $\R^d$} into $n$ convex regions such that each region would contain one of the disjoint convex hulls.

By general position, no
$\ell$ points of $X$ with $1\le\ell\le d$ lie on a common
 $(\ell-1)$-flat (affine subspace of
dimension~$\ell-1$). 
When we replace the points by balls, we make their radius 
small enough so that no
$\ell$ of these \emph{balls} are intersected by
any $(\ell-1)$-flat.
\begin{lemma}\label{convex general}  
	Let $P\subseteq \R^d$ be a finite set of points in general
        position, and let $\varepsilon>0$ be chosen such that no 
$\ell$ closed balls $B_\varepsilon(x)$ of radius~$\varepsilon$
centered at points $x$ from~$P$ with $1\le \ell \le d$ can be
intersected by a common $(\ell-1)$-flat.
	Suppose we are given an affine hyperplane $H\subseteq \R^d$ and a partition of $P=P^+\sqcup P^-$ satisfying
	\[
	P^+\subseteq\{x\in P : B_\varepsilon(x)\cap H^+\neq\emptyset\}
	\quad\text{and}\quad
	P^-\subseteq\{x\in P : B_\varepsilon(x)\cap H^-\neq\emptyset\},
	\]
	where $H^+$ and $H^-$ are the open half-spaces determined by $H$.  
	Then 
	\[
	\conv P^+\cap\conv P^-=\emptyset.
	\]
\end{lemma}
\begin{proof}
  The proof is based on the perturbation argument from  
  \cite[Proof of Lemma~2]{AA89}, which we make more
  explicit.

  We perform a reverse induction on the number $\ell$ of
  $\eps$-balls intersected by $H$, starting with the maximum possible
  value $\ell=d$ and proceeding downwards. 
  For the induction basis, when
  $\ell=d$, we choose points $q_1,\ldots,q_d$ in the intersection of
  the $d$ balls with $H$. We move them straight to the ball
  centers at constant speed, and we let the hyperplane $H$ through
  the points follow along. By our assumptions, $H$ is always
  uniquely defined throughout the motion, and it will not intersect any
  other ball at any time. When the points $q_i$ arrive, $H$ goes
  \emph{through} $d$ points of~$P$. We now perturb each point
  $q_i$ perpendicular to $H$ in the appropriate direction, so that the
  corresponding point of $P$ will be on the desired side, $H^+$ or
  $H^-$. The position of the remaining $n-d$ points with respect to $H$ is
  unchanged throughout this process, and it was correct from the beginning since $H$ did not
  intersect their balls.
  Thus, we have a hyperplane strictly separating the sets $P^+$ and
  $P^-$.  Consequently, $\conv P^+\cap \conv P^-=\emptyset$.

  Let us now consider the case that $H$ intersects $\ell<d$ balls.  As
  above, we choose points $q_1,\ldots,q_\ell$ from $H$ inside these
  balls, and we move them towards the ball centers. The
  motion of~$H$ is no longer uniquely defined, but since
  the moving points are never in degenerate position,
 they span an $(\ell-1)$-flat $L$ that moves continuously,
  %, we incrementally choose additional
%  stationary points $q_{\ell+1},\ldots,q_d$ from $H$, taking care that
%  $q_{k+1}$ does not lie on a $k$-flat through $q_{\ell+1},\ldots,q_k$
%  and any $\ell$ points from the $\ell$ balls.
  and we can continuously move the hyperplane $H$ while containing $L$.
%  keeping the moving points on~$H$.
  If, during this motion, $H$~intersects an additional ball, we have
  increased~$\ell$ and we proceed by induction. Otherwise, we
  arrive at a position where $H$ goes through $\ell$ points of $P$,
  and we perform the perturbation as above.  
\end{proof}

In order to assign boundary points to regions we will set up a flow network with a fractional flow;
from this we obtain an integer flow, which in turn will determine the assignment.

In a directed graph $D=(V,A)$ with a set of vertices $V$  and a set of arcs $A$, a \emph{flow} is a function 
$f\colon A\to\R$ that assigns a real number to each arc.
The \emph{excess} of the flow $f\colon  A\to\R$ at the vertex~$v$ of the graph $V$ is the difference
between the inflow and the outflow:
\[
	\excess(f,v) := \sum_{(u,v)\in A}f(u,v) - \sum_{(v,w)\in A}f(v,w)
\]
To obtain an integer flow from a fractional one, we will use the following statement.

\begin{proposition}\label{rounding-flows}
  Let $D=(V,A)$ be a directed graph, and let $F_-,F_+\colon A\to\Z$,
and $E_-,E_+\colon V\to\Z$ be integer-valued functions on the arcs and
on the vertices, respectively. 
		If there is some flow $f\colon A\to\R$ on $D$ such that  
		\[
		F_-(a)\leq f(a)\leq F_+(a)
		\quad\text{for }a\in A
		\qquad\text{and}\qquad
		E_-(v)\leq \excess(f,v)\leq E_+(v)
		\quad\text{for }v\in V,
		\]
		then there is also an integer flow $f'\colon A\to\Z$ that satisfies the same bounds. 
\end{proposition}

\begin{proof}
  This is a variation of the well-known integrality results on network flows.
Classical flow networks involve only a single vertex with negative
excess (source) and a single vertex with positive excess (sink),
conserving the flow at all other vertices. The network we consider has
several sources and sinks. Additionally, the excesses at these
vertices
are not fixed but allowed to vary within bounds,
and we have lower as well as upper
capacities on the arcs.

Such networks can be reduced to the classical situation by
standard transformations;
We sketch these transformations.
(See % for instance
\cite[Corollary
11.2i]{Schrijver:VolA} for an alternative approach via Hoffman's
circulation theorem.)
\begin{itemize}
\item 
%(A) 
First, any arc $(u,v)$ with a positive lower bound $F_-(u,v)$ and
upper bound $F_+(u,v)$ is replaced by a conventional arc
with nonnegative flow and
upper bound $F_+'(u,v) = F_+(u,v)-F_-(u,v)$.
 To compensate this
offsetting of the flow, the excesses have to
be adapted. We
 subtract
$F_-(u,v)$ from
 the excess bounds
$E_-(v)$ and
$E_+(v)$ at $v$
and add it
 to the excess bounds
$E_-(u)$ and
$E_+(u)$ at~$u$.
%  An arc $(u,v)$ with negative bounds can be treated by
% introducing the reverse arc $(v,u)$, but
% the network that we will use has no such arcs.
In the network that we will use,
there are no arcs
 with negative bounds.
Such an arc $(u,v)$ % with negative bounds 
could be treated by
introducing the reverse arc $(v,u)$.
\item %(B)~
To deal with the variation of the excess,
 we create an additional ``balancing sink'' $B$.
We fix the excess of each vertex at its lower bound. Thus,
a vertex $v$ with excess bounds
$E_-(v)$ and
$E_+(v)$
% (after the modifications of step~A)
 (as modified in the previous step)
 is declared to be
a sink with demand $E_-(v)$ if 
 $E_-(v)\ge 0$, or
a source with supply $-E_-(v)$ if 
 $E_-(v)<0$.
The excessive inflow at $v$ is then absorbed by
an arc $(v,B)$ of capacity
$E_+(v)-E_-(v)$.
We fix
the demand of $B$ to make the overall sum of demands equal to
the overall supply.
\item %(C)~
Finally, we add a
 new super-source $S$
with an arc to every source $v$, of capacity $F_+(S,v)$ equal to the
supply at~$v$,
 and a new super-sink $T$,
with an arc from every sink $v$ to~$T$, of capacity $F_+(v,T)$ equal to the
demand at~$v$.
\end{itemize}
Flows in the original graph~$D$ correspond, in the transformed network, to
classical flows from $S$ to $T$ that saturate all edges out
of~$S$.  Integrality is preserved throughout the transformation.
%
% introducing a new super-source
% and super-sink
% and modifying the network;  
% see for example Schrijver \cite[Chapter~11]{Schrijver:VolA} for such transformations.
% %; see for instance \cite[Corollary 11.2i]{Schrijver:VolA}.
        \end{proof}

\section{Proof of the main result}
\label{main-proof}
\begin{proof}[Proof of Theorem~\ref{equipart}]
	 Let $n$ and $d$ be positive integers, and let $X$ be a $d$-colored set of points in general position in $\R^d$.
	 Using the tools presented in Section \ref{toolsection}, we
         now prove our claim that we can partition $X$ into $n$ sets
         of size $\floor{\frac{|X|}{n}}$ or $\ceil{\frac{|X|}{n}}$
         that have pairwise disjoint convex hulls and simultaneously equipartition the color classes. 
	 The proof is done in several steps.

	\smallskip\noindent
	(1) \emph{From points to measures.}
	\\
	We replace each point $x\in X$ by a ball $B_\varepsilon(x)$ centered at $x$, with $\varepsilon>0$ a real number small enough such that no $\ell$-flat with $\ell<d$ intersects more than $\ell+1$ balls.
	With each ball centered at a point in~$X$, we associate a uniformly distributed measure of unit total mass. 
	For each $i\in\{1,\dots,d\}$ and for every measurable subset $A\subseteq\R^d$, let $\mu_i(A)$ be the total measure of balls centered at points in $X_i$ that is captured by $A$. 
	Clearly, $\mu_1,\dots,\mu_d$ are absolutely continuous finite measures on $\R^d$ with $\mu_i(\R^d)=|X_i|$. 
	According to Theorem~\ref{Spatzenbraten}, there exists a partition of~$\R^d$ into $n$ convex regions $C_1,\dots,C_n$ 
	which equipartitions the measures, that is, 
	\[\mu_i(C_j)=\frac{|X_i|}{n}\]
	for all $ i\in\{1,\dots,d\}$ and all $j\in\{1,\dots,n\}$.
	
	\smallskip\noindent
	(2) \emph{A directed graph of incidences.}
	\\
	In order to apply Lemma~\ref{convex general}, we show the existence of an 
	assignment of the points in $X$ to the $n$~regions
	$C_1,\dots,C_n$ such that for each point $x$ assigned to a region $C_j$, 
	$B_\varepsilon(x)$ intersects $C_j$, while in total $\lfloor\frac{|X|}{n}\rfloor$ or $\lceil\frac{|X|}{n}\rceil$ points are assigned to each region, with $\lfloor\frac{|X_i|}{n}\rfloor$ or $\lceil\frac{|X_i|}{n}\rceil$ of color class~$i$ for every $i\in\{1,\dots,d\}$. Such an assignment may be modeled as an integer flow from the points in $X$ to the regions in the partition, where each $x\in X$ has an outflow of $1$ and each region has an inflow of $\lfloor\frac{|X|}{n}\rfloor$ or $\lceil\frac{|X|}{n}\rceil$, the number of points assigned to it. To guarantee an equipartition of the color classes, we add a middle layer of vertices, one for each color and region, and set the constraints on these vertices and arcs accordingly.
	
	We define the directed graph $D=(V,A)$ with $V=X\sqcup Y\sqcup
	Z$,
	where  \[Y=\big\{\,y_i^j : 1\leq i \leq d,\ 1\leq j\leq n\,\big\}\]
	contains a vertex $y_i^j$ for each color~$i$ and each
	region $C_j$, and the set 
	$Z=\{C_1,\dots,C_n \}$ contains a vertex for each region.
	We have arcs from a point $x\in X$ to those vertices in $Y$ corresponding to the color of $x$ and the regions incident to the ball $B_\varepsilon(x)$ centered at $x$, as well as arcs from the vertices in $Y$ to their respective region in $Z$. More precisely, the set of arcs is
\begin{align*}
 A \coloneqq\ &\bigl\{\, (x,y_i^j) : 1 \leq i\leq d,\ 1\leq j\leq n,\ x\in X_i,\ B_\varepsilon(x)\cap C_j\neq 0\,\bigr\}\\
 &\cup\ \bigl\{ (y_i^j,C_j) : 1 \leq i\leq d,\ 1\leq j\leq n \,
		\bigr\}.
\end{align*}
\begin{figure}
	\centering
	\includegraphics[height=61mm]{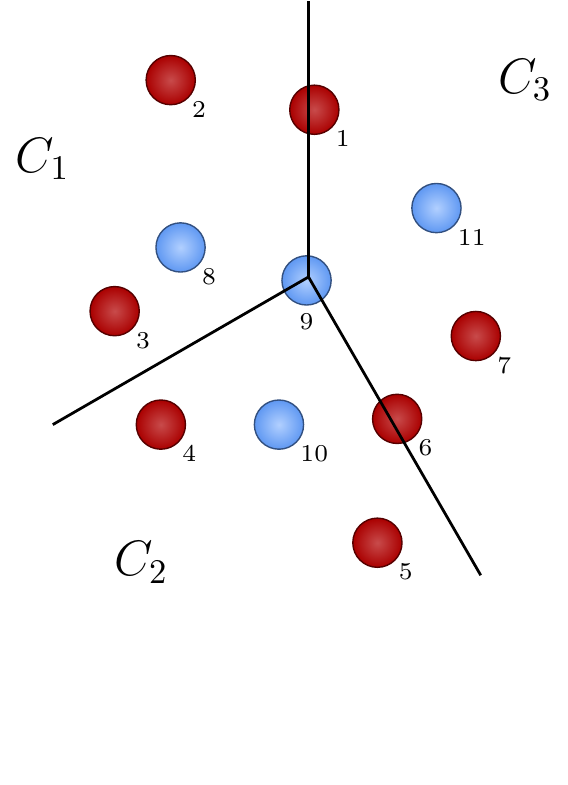}
	\qquad\qquad
	\includegraphics[height=75mm]{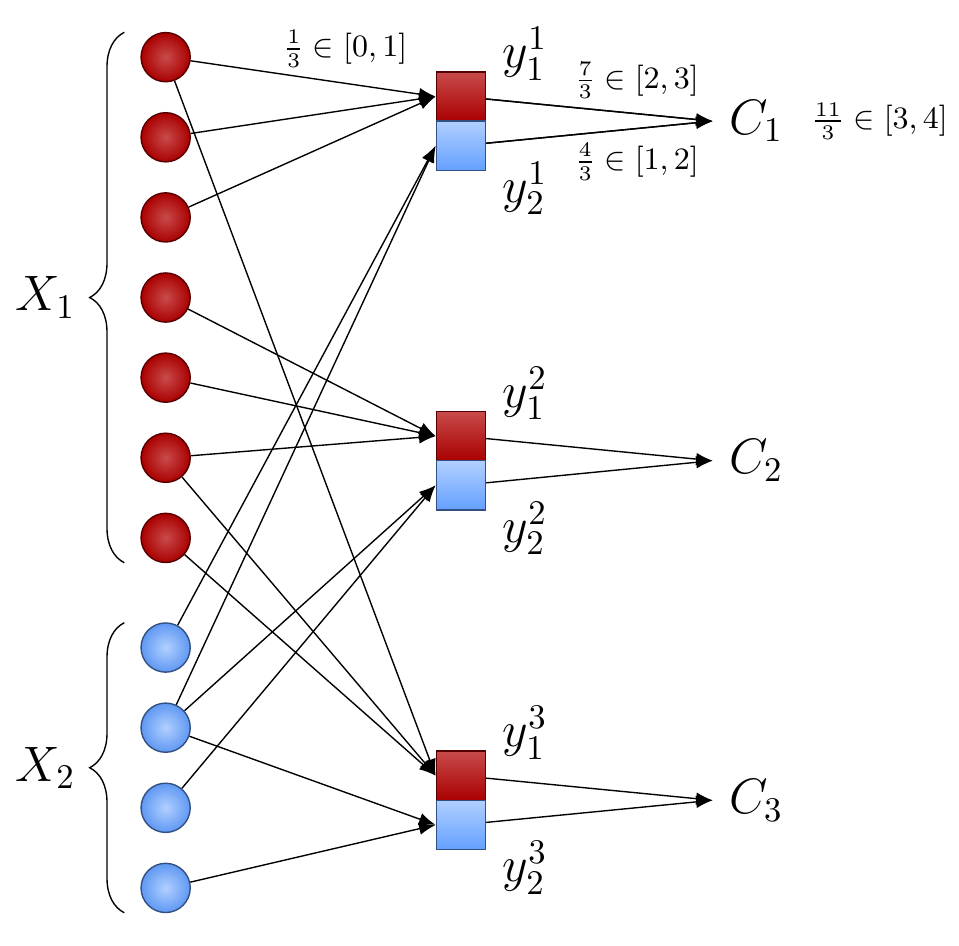}
	\caption{\small A configuration of 11 points/small balls
		with $d=2$ colors
		in $d=2$ dimensions,
		partitioned into $n=3$ regions,
		and the corresponding directed graph $D$ with some upper and lower bounds on the flow and its excess indicated as intervals.
		}
	\label{fig:small_example}
\end{figure}%
	For the vertices of $D$, we define lower bounds $E_-\colon V\to\Z$ and upper bounds $E_+\colon V\to\Z$ on the excess as follows:
	\begin{align*}
	\hskip 0,2\textwidth
 	E_-(x)&\coloneqq-1 &  &\qquad & E_+(x)&\coloneqq -1
	\hskip 0,2\textwidth
	\\
	 E_-(y)&\coloneqq0 & &\qquad & E_+(y)&\coloneqq 0\\
	 E_-(C_j)&\coloneqq\left\lfloor{\frac{|X|}{n}}\right\rfloor & &\qquad & E_+(C_j)&\coloneqq\left\lceil{\frac{|X|}{n}}\right\rceil
\\%	\end{align*}
\intertext{For the arcs of $D$, we define lower bounds $b_-\colon A\to\Z$ and upper bounds $b_+\colon A\to\Z$ as follows:
}
%	\begin{align*}
	%\hskip 0,2\textwidth
	b_-(x,y)&\coloneqq0 & &\qquad & b_+(x,y)&\coloneqq 1
	%\hskip 0,2\textwidth 
	\\
	b_-(y_i^j,C_j)&\coloneqq\left\lfloor{\frac{|X_i|}{n}}\right\rfloor & &\qquad & b_+(y_i^j,C_j)&\coloneqq\left\lceil{\frac{|X_i|}{n}}\right\rceil
	\end{align*}
	In all five cases, the %reader may check that 
the lower bounds don't exceed % are set at values that are smaller
                              % than or equal to 
the upper bounds.
	
	\smallskip\noindent
	(3) \emph{A fractional flow.}
	\\	
	We now construct a fractional flow $f\colon A\to\R$ by setting
	\[ 
	f(x,y_i^j)\coloneqq\mu_i(B_\varepsilon(x)\cap C_j) \text{\qquad and\qquad } f(y_i^j, C_j)\coloneqq\frac{|X_i|}{n}.
	\]
	The lower and upper constraints on the arcs are trivially satisfied, \[b_-(a)\leq f(a)\leq b_+(a) \text{\qquad for all \ } a\in A.\]
	With $\mu_i(B_\varepsilon(x))=1$ for all $x\in X_i$, we get 
	\[ 
	E_-(x)=-1=\excess(f,x)=-\sum_{j=1}^{n}\mu_i(B_\varepsilon(x)\cap C_j)=-1= E_+(x).
	\]
	With $\mu_i(C_j)=\frac{|X_i|}{n}=f(y_i^j,C_j)$ for a vertex $y_i^j\in Y$, the values yield
	\[ 
	E_-(y_i^j)= 0=\excess(f,y_i^j) = \sum_{x\in X} \mu_i(B_\varepsilon(x)\cap C_j) - f(y^j_i,C_j)  
	= 0 = E_+(y_i^j).
	\]
	Lastly, for  a $C_j\in Z$ we get
	\[ E_-(C_j)=\left\lfloor\frac{|X|}{n}\right\rfloor \leq \excess(f,C_j) = \sum_{i=1}^{d}f(y_i^j,C_j)= \frac{|X|}{n} \leq \left\lceil\frac{|X|}{n}\right\rceil = E_+(C_j),\]
	and consequently $E_-(v)\leq \excess(f,v)\leq E_+(v)$ for all $v\in V$.
	
	\smallskip\noindent
	(4) \emph{Back to geometry.}
	\\		
	From this fractional flow, Proposition~\ref{rounding-flows} produces an integer flow on $D$ that satisfies the constraints given by functions $b_-,b_+$ and~$E_-,E_+$.  
	This in turn gives an assignment of points into sets of size $\lfloor\frac{|X|}{n}\rfloor$ and~$\lceil\frac{|X|}{n}\rceil$, equipartitioning $X$. 
	The middle layer of $D$ ensures that each of the sets contains $\floor{\frac{|X_i|}{n}}$ or $\ceil{\frac{|X_i|}{n}}$ points from the color class $X_i$, resulting in a simultaneous equipartition of $X$ and all $d$ color classes.
	
	We now want that, for any two regions $C_j$ and $C_k$, the sets of points $P^+$ assigned to~$C_j$ and $P^-$ assigned to~$C_k$ have disjoint convex hulls.
	For each point $x$ assigned to a region, $B_\varepsilon(x)$ intersects that region, by the definition of the arc set $A$.
	We may therefore apply Lemma~\ref{convex general} to the set $P=P^+\sqcup P^-$  
	and conclude that the convex hulls of $P^+$ and $P^-$ are disjoint.
\end{proof}

\subsubsection*{Acknowledgements}
We are grateful to the four DCG referees for many useful comments and suggestions.

\bibliography{b}{}

\begin{thebibliography}{10}

\bibitem{AetAl10}
O.~Aichholzer, S.~Cabello, R.~Fabila-Monroy, D.~Flores-Pe\~{n}aloza, T.~Hackl,
  C.~Huemer, F.~Hurtado, and D.~R. Wood.
\newblock Edge-removal and non-crossing configurations in geometric graphs.
\newblock {\em Discrete Math. Theor. Comput. Sci. (DMTCS)}, 12:75--86, 2010.

\bibitem{AA89}
J.~Akiyama and N.~Alon.
\newblock Disjoint simplices and geometric hypergraphs.
\newblock In {\em Combinatorial Mathematics: Proc. of the Third International
  Conference, New York 1985}, volume 555 of {\em Annals of the New York Academy
  of Sciences}, pages 1--3, 1989.

\bibitem{BKS00}
S.~Bespamyatnikh, D.~Kirkpatrick, and J.~Snoeyink.
\newblock Generalizing ham sandwich cuts to equitable subdivisions.
\newblock {\em Discrete Comput. Geom.}, 24:605--622, 2000.

\bibitem{BZ14}
P.~V.~M. Blagojevi{\'{c}} and G.~M. Ziegler.
\newblock Convex equipartitions via {E}quivariant {O}bstruction {T}heory.
\newblock {\em Israel Journal of Mathematics}, 200:49--77, 2014.

\bibitem{HKV16}
A.~F. Holmsen, J.~Kyn\v{c}l, and C.~Valculescu.
\newblock Near equipartitions of colored point sets.
\newblock {\em Comput. Geom. Theory Appl.}, 65:35–42, 2017.

\bibitem{IUY00}
H.~Ito, H.~Uehara, and M.~Yokoyama.
\newblock $2$-dimension ham sandwich theorem for partitioning into three convex
  pieces.
\newblock In {\em Discrete and Computational Geometry: Japanese Conference,
  JCDCG'98 Tokyo, Japan, December 9--12, 1998. Revised Papers}, pages 129--157.
  Springer-Verlag, 2000.

\bibitem{KK16}
M.~Kano and J.~Kyn\v{c}l.
\newblock The hamburger theorem.
\newblock {\em Comput. Geom. Theory Appl.}, 68:167--173, 2018.

\bibitem{KSU14}
M.~Kano, K.~Suzuki, and M.~Uno.
\newblock Properly colored geometric matchings and $3$-trees without crossings
  on multicolored points in the plane.
\newblock In {\em Discrete and Computational Geometry and Graphs}, volume 8845
  of {\em Lecture Notes in Comput. Sci.}, pages 96--111. Springer-Verlag, 2014.

\bibitem{K11}
R.~Karasev.
\newblock Equipartition of several measures.
\newblock Preprint, June 2013, 11 pages,
  \href{https://arxiv.org/abs/1011.4762}{{arXiv:1011.4762v7}}.

\bibitem{KHA13}
R.~Karasev, A.~Hubard, and B.~Aronov.
\newblock Convex equipartitions: {T}he spicy chicken theorem.
\newblock {\em Geometriae Dedicata}, 170:263--279, 2014.

\bibitem{Sakai02}
T.~Sakai.
\newblock Balanced convex partitions of measures in {${\mathbb R}^2$}.
\newblock {\em Graphs and Combinatorics}, 18:169--192, 2002.

\bibitem{Schrijver:VolA}
A.~Schrijver.
\newblock {\em Combinatorial {O}ptimization---{P}olyhedra and {E}fficiency.
  {V}ol. {A}: Paths, Flows, Matchings}, volume~24 of {\em Algorithms and
  Combinatorics}.
\newblock Springer-Verlag, Berlin, 2003.

\bibitem{S12}
P.~Sober\'{o}n.
\newblock Balanced convex partitions of measures in {${\mathbb R}^d$}.
\newblock {\em Mathematika}, 58:71--76, 2012.

\end{thebibliography}
\bibliographystyle{abbrv}

\end{document}